\documentclass[12pt, final]{amsart}

\usepackage{geometry}
\usepackage{multirow}
\usepackage{graphicx}
\usepackage{float}
\usepackage{amssymb}
\usepackage{verbatim} 
\usepackage{amsmath}
\usepackage{amsthm}
\usepackage{color}
\usepackage{epstopdf}
\usepackage{epsfig}
\usepackage[all]{xy}
\usepackage{hyperref}

\newcommand{\vol}{{\rm vol}}
\newtheorem{defi}{Definition}
\newtheorem{thm}{Theorem}

\newtheorem{lem}{Lemma}

\begin{document}
\title[Loxodromic unit vector field on punctured spheres]{Loxodromic unit vector field on punctured spheres}

\author{Jackeline Conrado$^1$}
\author{Adriana V. Nicoli$^2$}
\author{Giovanni S. Nunes$^3$}

\thanks{The first and second authors are supported by a scholarship from the National Doctoral Program, CAPES-PROEX}

\address{Dpto. de Matem\'{a}tica, Instituto de Matem\'{a}tica e Estat\'{i}stica,
Universidade de S\={a}o Paulo, R. do Mat\={a}o 1010, S\={a}o Paulo-SP
05508-900, Brazil.}
\email{jconrado@usp.br, avnicoli@ime.usp.br}

\address{Dpto de Matem\'atica e Estat\'istica, Instituto de Matem\'{a}tica e F\'isica,
Universidade Federal de Pelotas, Rua Gomes Carneiro 1, Pelotas - RS
96001-970, Brazil}
\email{giovanni.nunes@ufpel.edu.br}


\subjclass[2019]{}

\begin{abstract} In these short notes we characterize the loxodromic unit vector fields on antipodally punctured Euclidean spheres as the only ones achieving a lower bound for the volume functional depending on the Poincar\'e indexes around their singularities.
\end{abstract}
\maketitle
\section{Introduction and statement of the results}

Let $M$ be a closed oriented Riemannian manifold and $\nabla$ the Levi Civita connection. Let $\left\{e_a\right\}^n_{a=1}$ be an orthonormal local frame in $M$ and $\vec{v}:M\to T^1M$ a unit vector field on $M$, where $T^1M$ is equipped with the Sasaki metric. The volume of $\vec{v}$ is defined on \cite{GW} and \cite{Jhonson} as
\begin{eqnarray}
\label{defvolume}
\vol(\vec{v}) &=& \int_M \Big(1+\sum_{a=1}^{n} \|\nabla_{e_a} \vec{v} \|^2  + \sum_{a_1<a_2} \| \nabla_{e_{a_1}} \vec{v}\wedge \nabla_{e_{a_2}} \vec{v} \|^2 + \cdots \nonumber \\
&\cdots& + \sum_{a_1<\cdots <a_{n-1}}\|\nabla_{e_{a_1}} \vec{v}\wedge \cdots \wedge \nabla_{e_{a_{n-1}}} \vec{v}  \|^2 
\Big)^{\frac{1}{2}}\nu, 
\end{eqnarray}
where $\nu$ denote the volume form for $\left\{e_a\right\}^n_{a=1}$.

Clearly $\vol(\vec{v}) \geq \vol(M)$ and $\vol(\vec{v})=\vol(M)$ if and only if $\vec{v}$ is a parallel field. Such vector fields are rare because if $M$ admits a unit parallel vector field, then $M$ is locally a Riemannian product. In spheres of even dimension, the vector field with isolated singularities arises naturally in the study of the volume functional. When $M$ is an antipodally punctured sphere, a relation between the volume and the Poincar\'e index of the vector field was established in \cite{BCJ}.

\begin{thm}\label{thmBCJ}(see Brito, Chac\'on, Johnson \cite{BCJ}) Let $M = \mathbb{S}^{n} \backslash \left\{N,S\right\}$, $n = 2$ or $3$, be the standard Euclidean sphere where two antipodal points $N$ and $S$ are removed. Let $\vec{v}$ be a unit vector field defined on $M$. Then,

(a) for $n=2$, $\vol(\vec{v})\geq \frac{1}{2} (\pi + |I_{\vec{v}}(N)| + |I_{\vec{v}}(S)| -2)\vol(\mathbb{S}^2)$,

(b) for $n=3$, $\vol(\vec{v})\geq (|I_{\vec{v}}(N)| + |I_{\vec{v}}(S)|)\vol(\mathbb{S}^3)$,

\noindent
where $I_{\vec{v}}(P)$ stands for the Poincar\'e index of $\vec{v}$ around $P$.
\end{thm}

It follows from the Theorem's 1 proof,  that the north-south unit vector field realizes this lower bound. A natural question arises: Is this the only one with this property? The answer is no. We prove that the loxodromic unit vector fields are unique with this property. Precisely:

\begin{thm}\label{principal}
 The lower bound for the volume functional on $\mathbb{S}^2\backslash \{ N, S\}$ is realized if, and only if, $\vec{v}$ is a loxodromic unit vector field.
\end{thm}

\section{Preliminaries}

Let $M = \mathbb{S}^{2} \backslash \left\{N,S\right\}$ be the standard Euclidean sphere whose two antipodal points $N$ and $S$ are removed. Denote by $g$ the usual metric of $\mathbb{S}^2$ induced from $\mathbb{R}^3$, and by $\nabla$ the Levi-Civita connection associated to $g$. Let $\vec{v}$ be an unit vector field in $M$. We compute the volume of $\vec{v}$ making use of a global orthornormal special frame and using the Levi-Civita conection $\nabla$ of $M$. Consider the oriented orthonormal local frame $\left\{ e_1 = \vec{v}^{\perp} , e_2 = \vec{v} \right\}$ on $M$ and its dual basis  $\left\{ \omega_1, \omega_2 \right\}$. The connection $1$-forms of $\nabla$ are $\omega_{ij}(X) = g(\nabla_{X}e_j, e_i)$ for $i,j = 1,2$ where $X$ is a vector in the corresponding tangent space. In dimension $2$, the volume (\ref{defvolume}) reduces to
\begin{eqnarray}\label{volreduces}
\vol(\vec{v}) = \int_{\mathbb{S}^2}{\sqrt{1+ \kappa^2 + \tau^2}}\nu,
\end{eqnarray}
where $\kappa=g(\nabla_{\vec{v}}\vec{v}, \vec{v}^{\perp})$ is the geodesic curvature of the integral curves tangent to $\vec{v}$ and $\tau = g(\nabla_{v^{\perp}}\vec{v}, \vec{v}^{\perp})$ is the geodesic curvature of the curves orthogonal to $\vec{v}$. Also,
\[\omega_{12}= \tau \omega_1 + \kappa\omega_2.\]

In a sphere, a  \textit{loxodromic} (or \textit{rhumb line}) is a curve crossing all parallels at the same angle.  Because of their nature, these curves spiral towards the poles, as can see at the Figure \ref{loxodromic}. Let us define \textit{loxodromic vector field}.

\begin{defi}
A \textbf{loxodromic unit vector field  in $M$} is an unit vector field that forms a constant angle along each parallel in $\mathbb{S}^2$.
\end{defi}
 
Observe that the north-south unit vector field is a loxodromic unit vector field in $M$.

 \begin{figure}[H]
	\centering
	\includegraphics[width=8cm]{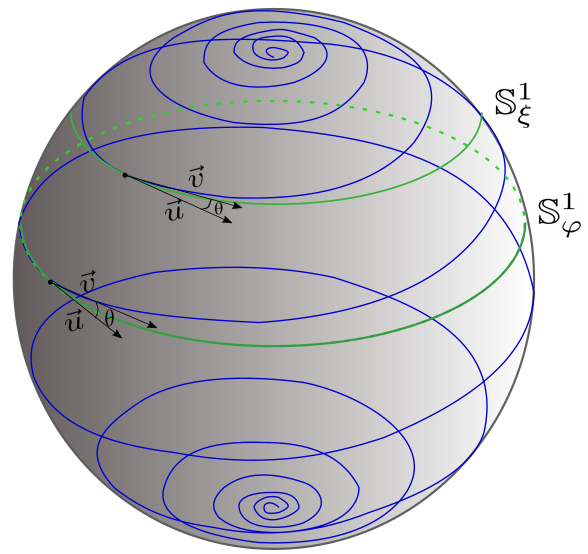}
	\caption{Loxodromic curve on $\mathbb{S}^2$}
	\label{loxodromic}
\end{figure}

 Let $\mathbb{S}^1_{\varphi}$ be the parallel of $\mathbb{S}^2$ at latitude $\varphi \in (-\frac{\pi}{2}, \frac{\pi}{2})$, figure \ref{Svarphi}. Let $\{\vec{u},\vec{n}\}$ be an oriented frame where $\vec{u}$ is tangent to $\mathbb{S}^1_{\varphi}$ and $\vec{n}$ is parallel to a south-north meridian. Let $\theta \in [0,\pi/2]$ be the oriented angle from $\vec{u}$ to $\vec{v}$. Then ${\vec{u}=\sin \theta \vec{v}^{\perp} + \cos \theta \vec{v} }$.
 
 \begin{figure}[H]
 	\centering
 	\includegraphics[width=8cm]{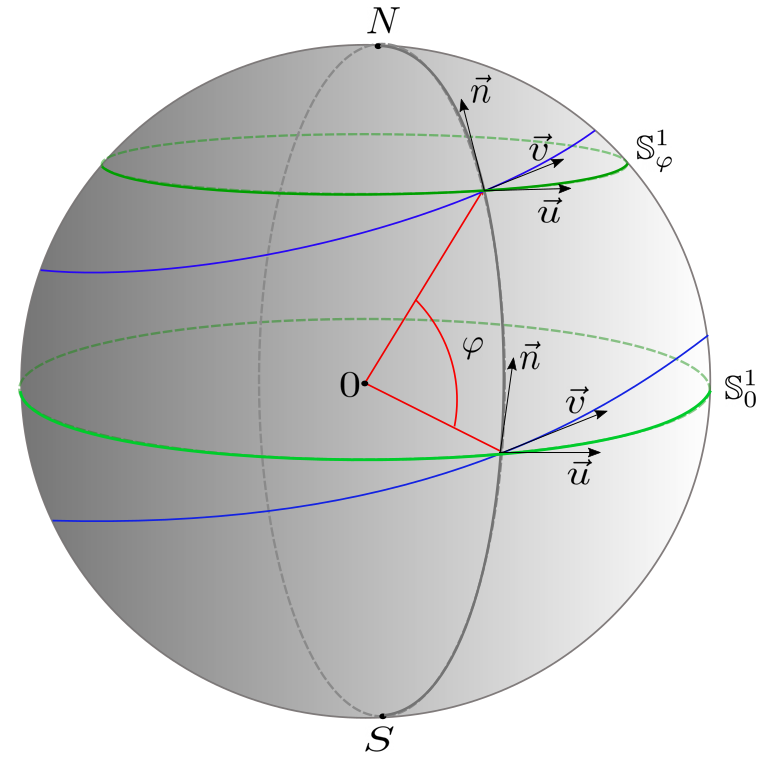}
 	\caption{$\mathbb{S}^1_{\varphi}$ be the parallel of $\mathbb{S}^2$ at latitude $\varphi \in (-\frac{\pi}{2}, \frac{\pi}{2})$.}
 	\label{Svarphi}
 \end{figure}
     
\section{Proof of the theorem}
Consider $\mathbb{S}^2 = \mathbb{S}^+ \cup \mathbb{S}^-$, where $\mathbb{S}^+$ and $\mathbb{S}^-$ are the northern and southern hemisphere, respectively. 

\begin{lem}\label{lemI} Let $\vec{v}$ be an unit vector field. If $\vec{v}$ satisfies $\left| \sin \varphi \right| = \sqrt{\kappa^2+\tau^2} \cos \varphi$ then, $\vec{v}$ realize the lower bound for the volume on $\mathbb{S}^2$. On the other hand, if $\vec{v}$ realize the lower bound for the volume on $\mathbb{S}^2$, then $\vec{v}$ satisfies
\begin{eqnarray}\label{cond_sharpness}
i)\left| \sin \varphi \right| = \sqrt{\kappa^2+\tau^2} \cos \varphi \hspace{0.4cm} \mbox{and} \hspace{0.4cm} ii) \kappa\sin \theta = \tau \cos \theta.
\end{eqnarray}
\end{lem}
\begin{proof}
Let $\vec{v}$ be satisfying  $\left| \sin \varphi \right| = \sqrt{\kappa^2+\tau^2} \cos \varphi$.\\
We analyze  two cases: north and south. 

In the first case consider the northern hemisphere $\mathbb{S}^+$.

Remember the general inequality, $\sqrt{a^2 + b^2} \geq |a\cos \beta + b \sin \beta|$. If $b\cos \beta = a \sin \beta$ then,
\begin{equation}\label{igualdade}\sqrt{a^2 + b^2} = \left|a\cos \beta + b \sin \beta \right|. 
\end{equation}
for any $a$, $b$, $\beta \in \mathbb{R}$.

Considering the positive part of $\left| \sin \varphi \right| = \sqrt{\kappa^2+\tau^2} \cos \varphi$ and (\ref{igualdade}), we have 
\[\sqrt{1+\kappa^2+\tau^2}=|\cos{\varphi}+\sqrt{\kappa^2+\tau^2}\sin{\varphi}|.\]

As we are in the $\mathbb{S}^+$, i.g. $0 \leq \varphi < \pi/2$,  implies that 
\[\sqrt{1+\kappa^2+\tau^2}=\cos{\varphi}+\sqrt{\kappa^2+\tau^2}\sin{\varphi},\]
and by the hypotheses, we get
\begin{equation}\label{cond_sharpnessI}
\sqrt{\kappa^2+\tau^2} = \tan{\varphi}.
\end{equation}


Then
\begin{equation}\label{igualdadevolume}
\sqrt{1 + \kappa^2+\tau^2} = \cos \varphi + \sqrt{\kappa^2+\tau^2} \sin \varphi = \cos \varphi + \tan \varphi \sin \varphi 
\end{equation}
where $0 \leq \varphi < \pi/2.$

 Denote by $\nu'$ the induced volume form to ${\mathbb{S}^1}_{\varphi}$. From (\ref{volreduces}) and (\ref{igualdadevolume}) we conclude
\begin{equation}\label{vol_hemis_norte}
\vol(\vec{v})_{|_{\mathbb{S}^+}} = \int_{\mathbb{S}^+}{\cos \varphi + \sqrt{\kappa^2 + \tau^2} \sin \varphi \nu} 
\end{equation}
\[
\hspace{2.3cm}= \int_0^{\frac{\pi}{2}}{ \int_{\mathbb{S}^1_{\varphi}} \cos \varphi  + \tan \varphi \sin \varphi \nu'd\varphi}
\]
\[
\hspace{2.1cm}= 2\pi \int_0^{\frac{\pi}{2}}{ \cos^2 \varphi + \sin^2 \varphi d\varphi} = \pi^2.
\]

For the second case, consider the southern hemisphere $\mathbb{S}^{-}.$
 
In this part $-\pi/2 < \varphi \leq 0$, (\ref{igualdade}) together with the negative part of $\left| \sin \varphi \right| = \sqrt{\kappa^2+\tau^2} \cos \varphi$ we have 
\[\sqrt{1+\kappa^2+\tau^2}= \left| \cos{\varphi}-\sqrt{\kappa^2+\tau^2}\sin{\varphi}\right|.\]
Observe that $\sin \varphi \leq 0$ and $\cos \varphi \geq 0$ implies
\begin{equation}\label{}
\left| \cos{\varphi}-\sqrt{\kappa^2+\tau^2}\sin{\varphi}\right| = \cos \varphi -\sqrt{\kappa^2+\tau^2}\sin{\varphi}. 
\end{equation}
and
\[ \sqrt{\kappa^2+\tau^2} = \left| \tan \varphi\right| = - \tan \varphi.\]
We attain
\begin{equation}\label{vol_hemis_sul}
\vol(\vec{v})_{|_{\mathbb{S}^-}} = \int_{\mathbb{S}^-}{\cos \varphi - \sqrt{\kappa^2 + \tau^2} \sin \varphi \nu} 
\end{equation}
\[
\hspace{2.6cm}= \int_{\frac{-\pi}{2}}^0{ \int_{\mathbb{S}^1_{\varphi}} \cos \varphi  - \left|\tan \varphi\right| \sin \varphi \nu'd\varphi}
\]
\[
\hspace{3.1cm}= \int_{\frac{-\pi}{2}}^0{ \int_{\mathbb{S}^1_{\varphi}} \cos \varphi  - (-\tan \varphi) \sin \varphi \nu'd\varphi}
\]
\[
\hspace{3.9cm}= 2\pi \int_{\frac{-\pi}{2}}^0{ \cos^2 \varphi + \tan \varphi \sin \varphi \cos\varphi d\varphi} = \pi^2.
\]

From (\ref{vol_hemis_norte}) and (\ref{vol_hemis_sul}) we get the volume of $\vec{v}$ is
$
\vol(\vec{v})_{|_{\mathbb{S}^2}} = 2\pi^2  = \frac{\pi}{2} \vol(\mathbb{S}^2). 
$
We conclude that $\vec{v}$ realizes the lower bound for the volume on $\mathbb{S}^2$.

Notice that the integral of $i^*\omega_{12}$, where $i$ is the inclusion of $\mathbb{S}^1_{\varphi}$ in $\mathbb{S}^2$, give us the indices of singularities (see \cite{Manfredo}) and, in our case,  $I_{\vec{v}}(N)=I_{\vec{v}}(S)=1$.

We now show that the $\vec{v}$ satisfies the conditions i) and ii) if, $\vec{v}$ realizes the lower bound for the volume on $\mathbb{S}^2$.
 
It follows from the proof of Theorem's 1 in \ref{thmBCJ} that the condition
\[ \sqrt{1 + \kappa^2 + \tau^2} = \left|\cos \varphi + \sqrt{\kappa^2 + \tau^2} \sin \varphi \right| = \left|\cos \varphi + \left|\kappa \cos \theta + \tau \sin \theta \right| \sin \varphi \right|.\]
is a consequence of $\vol(\vec{v})=\frac{\pi}{2}\vol(\mathbb{S}^2)$.

In the case  $0 \leq \varphi < \pi/2$, the first equality implies $\sin \varphi \ = \sqrt{\kappa^2+\tau^2} \cos \varphi $. 
If $\ -\pi/2 < \varphi \leq 0$, in $\mathbb{S}^-$, then $0 \leq \ -\varphi < \pi/2$ so $ \sin(-\varphi)  = \sqrt{\kappa^2+\tau^2} \cos(-\varphi) $  or $\ -\sin \varphi \ = \sqrt{\kappa^2+\tau^2} \cos \varphi$, therefore $\left| \sin \varphi \right| = \sqrt{\kappa^2+\tau^2} \cos \varphi$ for $-\pi/2< \varphi < \pi/2$.
From second equality we get the condition ii).  
\end{proof}

\begin{lem}\label{lemII} 
Let $M$ be a $2$-dimensional Riemannian manifold. If $\vec{v}$ be a unit vector field in $M$ and $\{\vec{u},\vec{n}\}$ be an oriented frame. Then 
\[\kappa = -\theta_{\vec{v}} - \cos \theta g(\nabla_{\vec{u}}{\vec{u}}, \vec{n}) \hspace{0.4cm}\mbox{and}\hspace{0.4cm} \tau = -\theta_{\vec{v}^{\perp}} + \sin \theta g(\nabla_{\vec{u}}{\vec{n}}, \vec{n}).\]
\end{lem}
\begin{proof}
Computing $\kappa = g(\nabla_{\vec{v}}\vec{v},\vec{v}^{\perp})$. Without loss of generality we may assume $\theta \in (0, \pi/2]$, so
\[\vec{v} = \cos \theta \vec{u} + \sin \theta \vec{n} \hspace{0.4cm}\mbox{and}\hspace{0.4cm} \vec{v}^{\perp} = \sin \theta\vec{u} - \cos \theta\vec{n}.\]
Thus,
\[ 
\nabla_{\vec{v}} \vec{v} =  \left[ \cos^2 \theta\nabla_{\vec{u}} \vec{u} + \cos \theta\sin \theta\nabla_{\vec{n}} \vec{u} + \vec{v}(\cos \theta)\vec{u}\right] + \left[ \cos \theta\sin \theta\nabla_{\vec{u}} \vec{n} + \sin^2 \theta\nabla_{\vec{n}} \vec{n} + \vec{v}(\sin \theta)\vec{n} \right]
\] 
and
\[
g\left( \nabla_{\vec{v}} \vec{v}, \vec{v}^{\perp} \right) =  g\left( \left[ \cos^2 \theta \nabla_{\vec{u}} \vec{u} + \cos \theta \sin \theta \nabla_{\vec{n}} \vec{u} + \vec{v}(\cos \theta)\vec{u}\right], \vec{v}^{\perp} \right)\]
\[ + g\left( \left[ \cos \theta\sin \theta \nabla_{\vec{u}} \vec{n} + \sin^2 \theta \nabla_{\vec{n}} \vec{n} + \vec{v}(\sin \theta)\vec{n} \right], \vec{v}^{\perp} \right)
\]
\[
= \underbrace{g\left( \cos^2 \theta \nabla_{\vec{u}} \vec{u}, \vec{v}^{\perp} \right)}_{a} + \underbrace{g\left( \cos \theta \sin \theta \nabla_{\vec{n}} \vec{u}, \vec{v}^{\perp} \right)}_{b} + \underbrace{g\left( \vec{v}(\cos \theta)\vec{u}, \vec{v}^{\perp} \right)}_{c} + \underbrace{g\left( \cos \theta\sin \theta \nabla_{\vec{u}} \vec{n}, \vec{v}^{\perp} \right)}_{d}\] \[ + \underbrace{g\left( \sin^2 \theta \nabla_{\vec{n}} \vec{n}, \vec{v}^{\perp} \right)}_{e} + \underbrace{g\left( \vec{v}(\sin \theta)\vec{n}, \vec{v}^{\perp} \right)}_{f},
\]
where
\[
a = g\left( \cos^2\theta \nabla_{\vec{u}} \vec{u}, \vec{v}^{\perp} \right) = \cos^2\theta\sin \theta g\left( \nabla_{\vec{u}}\vec{u}, \vec{u}\right) - \cos^3\theta g\left( \nabla_{\vec{u}}\vec{u}, \vec{n}\right) = - \cos^3\theta g\left( \nabla_{\vec{u}}\vec{u}, \vec{n}\right) .
\]

\[
b = g\left( \cos\theta\sin \theta \nabla_{\vec{n}} \vec{u}, \vec{v}^{\perp} \right) = \cos\theta\sin^2 \theta g\left( \nabla_{\vec{n}}\vec{u}, \vec{u}\right) - \cos^2\theta\sin \theta g\left( \nabla_{\vec{n}}\vec{u}, \vec{n}\right).
\]

\[
c = g\left( \vec{v}(\cos\theta)\vec{u}, \vec{v}^{\perp} \right) = \vec{v}(\cos\theta)\sin \theta = - \theta_{\vec{v}} \sin^2 \theta.
\]

\[
d = g\left( \cos\theta\sin \theta \nabla_{\vec{u}} \vec{n}, \vec{v}^{\perp} \right) = \cos\theta\sin^2 \theta g\left( \nabla_{\vec{u}}\vec{n}, \vec{u}\right) - \cos^2\theta\sin \theta g\left( \nabla_{\vec{u}}\vec{n}, \vec{n}\right).
\]

\[
e = g\left( \sin^2 \theta \nabla_{\vec{n}} \vec{n}, \vec{v}^{\perp} \right) =  \sin^3 \theta g\left( \nabla_{\vec{n}}\vec{n}, \vec{u}\right) - \cos\theta\sin^2 \theta g\left( \nabla_{\vec{n}}\vec{n}, \vec{n}\right) =  \sin^3 \theta g\left( \nabla_{\vec{n}}\vec{n}, \vec{u}\right).
\]

\[
f = g\left( v(\sin \theta)\vec{n}, \vec{v}^{\perp} \right) = \vec{v}(\sin \theta)\cos\theta = - \theta_v\cos^2 \theta.
\]
By orthogonality between $\vec{u}$ and $\vec{n}$, 
\[g(\nabla_{\vec{n}}\vec{n},\vec{u}) = g(\nabla_{\vec{n}}\vec{u},\vec{u})=0.\]
Then,
\[
g\left( \nabla_{\vec{v}} \vec{v}, \vec{v}^{\perp} \right) = - \cos^3\theta g\left( \nabla_{\vec{u}}\vec{u}, \vec{n}\right) -\theta_{\vec{v}}\sin^2 \theta + \cos \theta\sin^2 \theta g\left( \nabla_{\vec{u}}\vec{n}, \vec{u}\right) - \theta_{\vec{v}}\cos^2 \theta.
\]
\[
= -\theta_{\vec{v}} + \left[ (-\cos^3 \theta - \cos \theta \sin^2 \theta) g\left( \nabla_{\vec{u}}\vec{u}, \vec{n} \right)\right] = -\theta_{\vec{v}} + \left[ -\cos \theta g\left( \nabla_{\vec{u}}\vec{u}, \vec{n} \right)\right].
\]
Therefore,
\[
 \kappa  = -\theta_{\vec{v}} - \cos \theta g \left( \nabla_{\vec{u}}\vec{u}, \vec{n} \right).
\]  	

A similar computation leads to 
\[\tau = -\theta_{\vec{v}^{\perp}} + \sin \theta g(\nabla_{\vec{u}}{\vec{n}}, \vec{n}).\]
\end{proof}

\textbf{The proof of Theorem \ref{principal}}: We begin by proving if $\vec{v}$ realizes the lower bound for the volume on $\mathbb{S}^2$ then, $\vec{v}$ is a loxodromic unit vector field.
If $0 \leq \varphi < \pi/2$, from Lemma \ref{lemI} we have  
\[i) \sin \varphi  = \sqrt{\kappa^2+\tau^2} \cos \varphi \hspace{0.4cm} \mbox{and} \hspace{0.4cm} ii) \kappa\sin \theta = \tau \cos \theta.
\]
By condition ii)  
\[\hspace{0.9cm}\theta \neq \pi/2 \hspace{0.2cm} \Rightarrow \hspace{0.2cm} \tau = \kappa \tan{\theta},\]
\[\theta = \pi/2  \hspace{0.2cm} \Rightarrow \hspace{0.2cm}  \kappa = 0.\] Considering $\theta \neq \pi/2 $, we have
\begin{equation}\label{secante}
\sqrt{\kappa^2+\tau^2}=\sqrt{\kappa^2 + \kappa^2\tan^2{\theta}}=\sqrt{\kappa^2(1+ \tan^2 \theta)}=  \left|\kappa\right| \sqrt{\sec^2{\theta}} = \left|\kappa\sec{\theta}\right| =  \frac{\left|\kappa\right|}{\left|\cos \theta\right|}.
\end{equation}
Finally, from (\ref{cond_sharpnessI}) and (\ref{secante}) we write the geodesic curvature in the integral curves of $\vec{v}$ as
\begin{equation}\label{cond_sharpnessII}
 \left|\kappa \right|= \cos{\theta}\tan{\varphi}.    
\end{equation}
Observe that, if $\theta = \pi/2$ then $|\kappa|=0$ and  $\tau =|\tan{\varphi}|$.  
 For $\theta \neq \pi/2 $, $\kappa = \cos{\theta}\tan{\varphi}$. So
\begin{equation}\label{kappatau}
\kappa = \cos{\theta} \tan{\varphi} \hspace{0.3cm}\Rightarrow\hspace{0.3cm} \tau = \sin{\theta} \tan{\varphi}
\end{equation}
Note that the case is similar if we consider $\theta \in (\pi/2,\pi].$

By Lemma \ref{lemII} the geodesic curvature of curves integral tangent to $\vec{v}$ is $-\theta_{\vec{v}}-\cos{\theta}g(\nabla_{\vec{u}}\vec{u}, \vec{n})$. 

Observe that $g(\nabla_{\vec{u}}\vec{u}, \vec{n}) = -\tan{\varphi}$, in that way
\begin{equation}\label{curgeodesica}
       \kappa = -\theta_{\vec{v}}+\cos{\theta}\tan{\varphi}.
\end{equation}
Using (\ref{kappatau}) in (\ref{curgeodesica}) we conclude that $-\theta_{\vec{v}}=0$. So the unit vector field $\vec{v}$ is a loxodromic unit vector field.

 Now, consider $\vec{v}$ as a loxodromic unit vector field. From Lemma \ref{lemII} we have  
\[\kappa  = -\theta_{\vec{v}} - \cos \theta g \left( \nabla_{\vec{u}}\vec{u}, \vec{n} \right) \hspace{1cm}\mbox{and}\hspace{1cm} \tau = -\theta_{\vec{v}^{\perp}} + \sin \theta g(\nabla_{\vec{u}}{\vec{n}}, \vec{n}).\]
Observe that, 
\[g(\nabla_{\vec{u}}\vec{u}, \vec{n}) = -\tan{\varphi}\hspace{1cm}\mbox{and}\hspace{1cm}g(\nabla_{\vec{u}}\vec{n}, \vec{n}) = \tan{\varphi},\]
in that way
\[       \kappa = -\theta_{\vec{v}}+\cos{\theta}\tan{\varphi}\hspace{1cm}\mbox{and}\hspace{1cm}\tau = -\theta_{\vec{v}^{\perp}} + \sin \theta \tan{\varphi},\]
since $\vec{v}$ is a loxodromic unit vector field, then $-\theta_{\vec{v}} = -\theta_{\vec{v}^{\perp}} = 0.$ Therefore,
\begin{equation}\label{curvaturasI}
       \kappa = \cos{\theta}\tan{\varphi}\hspace{2cm}\mbox{and}\hspace{2cm}\tau = \sin \theta \tan{\varphi}.
\end{equation}
Substituting (\ref{curvaturasI}) in $\sqrt{\kappa^2 + \tau^2}$ we obtain
\[\sqrt{\tan^2{\varphi}} = \left|\tan{\varphi} \right|.\]
It implies,
\[\left|\sin \varphi \right| = \sqrt{\kappa^2 + \tau^2}\cos \varphi,\]
for $ - \pi/2 < \varphi < \pi/2.$
From Lemma \ref{lemI}, we conclude that $\vec{v}$ realizes the lower bound for the volume on $\mathbb{S}^2$. 

\textbf{Acknowledgements}  
We would like to thank Fabiano Brito for atracting our attention to the subject of this paper. His comments, suggestions, and encouragements helped shape this article. The first and second authors were financed by Coordena\c c\~ao de Aperfei\c coamento de Pessoal de N\'ivel Superior - Brasil (CAPES) - Finance Code $001$.

\end{document}